\documentclass[12pt]{article}

\usepackage{amsmath,amsfonts,amssymb,amsthm,amscd}
\usepackage{extarrows}

\title{Definable Topological Dynamics for Trigonalizable  Algebraic Groups over $\Q$}
\date{\today}

\author{Ningyuan Yao\\ Fudan University}

\newtheorem{Theorem}{Theorem}[section]
\newtheorem{Thm}{Theorem}[section]
\newtheorem{Prop}[Thm]{Proposition}
\newtheorem{Def}[Thm]{Definition}
\newtheorem{Rmk}[Thm]{Remark}
\newtheorem{Lemma}[Thm]{Lemma}
\newtheorem{Cor}[Thm]{Corollary}
\newtheorem{Fact}[Thm]{Fact}

\newtheorem{Question}[Thm]{Question}

\newtheorem*{Claim}{Claim}

\newtheorem*{Conj1}{Conjecture 1}
\newtheorem*{Conj1'}{Conjecture 1'}
\newtheorem*{Conj2}{Conjecture 2}

\newcommand{\Q}{{\mathbb Q}_p}

\newcommand{\N}{\mathbb N}

\newcommand{\M}{\mathbb M}

\newcommand{\K}{\mathbb K}
\newcommand{\Ga}{{\mathbb G}_a}
\newcommand{\Gm}{{\mathbb G}_m}

\newcommand{\sq}{\subseteq}

\newcommand{\lra}{\longrightarrow}

\DeclareMathOperator{\Th}{Th}

\DeclareMathOperator{\tp}{tp}

\DeclareMathOperator{\cl}{cl}

\DeclareMathOperator{\alg}{{alg}}

\begin{document}

\maketitle
\begin{abstract}
 We study the flow $(G(\Q), S_{G}(\Q))$ of trigonalizable algebraic group acting on its type space, focusing on the problem raised in \cite{P-Y} of whether weakly generic types coincide with almost periodic types if the group has global definable $f$-generic types, equivalently whether the union of minimal subflows of a suitable type space is closed. We will give a description of of $f$-generic types of trigonalizable algebraic groups, and prove that every $f$-generic type is almost periodic.
\end{abstract}

\section{Introduction}
In the recent years there has been growing interest in the interaction between topological
dynamics and model theory. This approach was introduced by Newelski \cite{Newelski},  then developed by a number papers, include \cite{Pillay-TD}, \cite{G. Jagiella}, \cite{CS}, \cite{Y-L} and \cite{P-Y}, and now called definable topological dynamics. Definable topological dynamics studies the action of a group $G$ definable in a structure $M$ on its type space $S_G(M)$ and tries to link the invariants suggested by  topological dynamics (e.g enveloping semigroups,  minimal subflows, Ellis groups...) with model-theoretic invariants. For example, in the case when $Th(M)$ is stable, the enveloping semigroups coincide with the type space of $G$, the minimal subflows coincide with the space of generic types, and Ellis groups also coincide with the space of generic types.

In the original paper \cite{Newelski}, almost periodic types are one of new object suggested by the topological dynamics point of view. A type $p\in S_G(M)$ is almost periodic if the closure of its $G$-orbit is a minimal subflow. Another new object, weakly generic formulas and types were introduced in \cite{N-P} as a substitute for generic formulas and types, since generic types may not always exists in $NIP$ environment. Briefly, a definable set (or formula) $X\sq G$ is weakly generic if there exists some nongeneric definable $Y\sq G$ such that  $X\cup Y$ is generic, where a definable set is generic if finitely many translates cover the whole group.

A nice observation in \cite{Newelski} was that the class of weak generic types in $S_G(M)$
is precisely the closure of the class of almost periodic types. When $G$ is stable, the two classes coincide. So one may naturally ask that whether the two classes coincide in some tame unstable theories. With the assumption that the theory $T$ has $NIP$ and The group is definably amenable,  \cite{CS}  gives a positive answer for Ellis group conjecture asked first by Newelski \cite{Newelski} and  then formulated precisely by Pillay \cite{Pillay-TD}. Moreover, \cite{CS} proved that weakly generic types coincide with $f$-generic types, which happens in stable context. These results  make it reasonable to consider definably amenable  groups  as  the ``stable-like" groups in NIP environment.

So Question 3.35 of \cite{CS} asked that wether weakly generic types coincide with almost periodic types for definably amenable groups in $NIP$ environment. One of the main results of \cite{P-Y} saying that the answer in negative even for $o$-minimal case. On the other hand, \cite{P-Y} proved that any \emph{definable} weakly generic type is almost periodic for definably amenable groups in $NIP$ environment. Based on this positive result, \cite{P-Y} conjectured that
\begin{Conj1}\label{Conjec1}
Let $G$ be a definably amenable group definable in an $NIP$ structure $M$. If $G$ has a global definably weakly generic type, then $p\in S_G(M)$ is weakly generic if and only if it is almost periodic.
\end{Conj1}
Assuming NIP, we say that a group has \emph{dfg} if it is definably amenable and has a global definable $f$-generic (or equivalently, weakly generic) type. We could restate the conjecture as follows:
\begin{Conj1'}\label{Conjec1}
Let $G$ be a group definable in an $NIP$ structure $M$. If $G$ has a \emph{dfg}, then $p\in S_G(M)$ is weakly generic if and only if it is almost periodic.
\end{Conj1'}

The aim of this paper is to prove the Conjecture \ref{Conjec1} in $p$-adic field $\Q$. The advantage of working in the $p$-adic case is that we have a good understanding of \emph{dfg} groups. A recent result of \cite{P-Y2} showing that any \emph{dfg} group  $G$  is eventually trigonalizable over $\Q$. Namely, there exist a finite index subgroup $A\leq G$, a finite subgroup $A_0\leq A$, and  a trigonalizable algebraic group $H$ such that $A/A_0$ is isomorphic to a open subgroup group of $H$. As all trigonalizable algebraic groups are \emph{dfg} (See Corollary \ref{G-is-dfg-G00}), we see that trigonalizable algebraic groups should be the ``maximal" ones among the class of \emph{dfg} groups. Since \cite{P-Y2} is  not avaliable online, we refer to the results from \cite{P-Y2} as conjectures. In this paper, we focus on trigonalizable algebraic groups.

We now highlight our main result as follows:
\begin{Theorem}\label{Main theorem}
Let $N$ be any elementary extension of $(\Q,+,\times,0,1)$, $G$  a trigonalizable algebraic group over $\Q$. Then $p\in S_G(N)$ is weakly generic if and only if it is almost periodic.
\end{Theorem}
The Theorem \ref{Main theorem} partially answered Conjecture \ref{Conjec1}. But with following conjecture
\begin{Conj2}\label{Conjec2}\cite{P-Y2}
Let $G$ be a definable group over $\Q$ with \emph{dfg}. Then there is an trigonalizable algebraic $H$ over $\Q$, a finite index subgroup $A\leq G$, and a finite-to-one homomorphism $f: A\longrightarrow G$ such that  $f(A)\leq H$ has finite index.
\end{Conj2}
 we could conclude from Theorem \ref{Main theorem} that Conjecture \ref{Conjec1} holds in $p$-adic case.

We will assume a basic knowledge of model theory. Good references are \cite{Pzt-book} and \cite{M-book}. Let $\mathbb T$ be a complete theory with infinite models. Its language is $L$ and $\M$ is the monster model, in which every type over a small subset  $A\subseteq \M$ is realized, where ``small" means $|A|<|\M|$.  $M,N, M', N'$ will denote the small elementary submodels of $\M$.  By $x,y,z$ we mean arbitrary $n$-variables and $a,b,c\in \M$  the $n$-tuples in $\M^n$ with $n\in \N$. every formula is an $L_\M$-formula. For an $L_M$-formula $\phi(x)$, $\phi(M)$ denote the definable subset of $M^{|x|}$ defined by $\phi$, and a set $X\subseteq M^n$ is  definable  if  there is an $L_M$-formula $\phi(x)$  such that $X=\phi(M)$. If $\bar X\subseteq \M^n$ is definable, defined with parameters from $M$, then $\bar X(M)$ will denote $\bar X\cap M^n$, the realizations from $M$, which is clearly a definable subset of $M^n$. Suppose that $X\subseteq \M^n$ is a definable set, defined with parameters from $M$, then we write $S_X(M)$ for the space of complete types concentrating on $X(M)$. We use freely basic notions of model theory such as definable type, heir, coheir, .... The book \cite{Pzt-book} is a possible source. Let $A,B$ be subsets of $\M$, and $p\in S(A)$, by $p|B$ we mean an heir of $p$ over $B$ if $A\subseteq B$, and the restriction of $p$ to $B$ if $A\supseteq B$.

The paper is organized as follows. In the rest of this introduction we recall precise definitions and results  from earlier papers, relevant to our results. In section 2.1, we will prove some general results for the closures of the orbit of a global type. In Section 2.2, we will characterise the type-definable connected component of a trigonalizable algebraic groups, and show that every trigonalizable algebraic group has \emph{dfg}.  Section 2.3 contains the main results of the paper, we give an description of the $f$-generic types of  trigonalizable algebraic groups, and showing that every  $f$-generic type is almost periodic.

\subsection{Topological dynamics and Definable Groups}

Our reference for (abstract) topological dynamics is \cite{Auslander}. Given a (Hausdorff) topological group $G$, by a $G$-flow mean a continuous action $G\times X\to X$ of $G$ on a compact (Hausdorff) topological space $X$. We sometimes write the flow as $(X,G)$. Often it is assumed that there is a dense orbit, and sometimes a $G$-flow $(X,G)$ with a distinguished point $x\in X$ whose orbit is dense is called a $G$-ambit.

In spite of $p$-adic algebraic groups being nondiscrete topological groups, we will be treating them as discrete groups so as to have their actions on type spaces being continuous. So in this background section we may assume $G$ to be a discrete group, in which case a $G$-flow is simply an action of $G$ by homeomorphisms on a compact space $X$.

By a subflow of $(X,G)$ we mean a closed $G$-invariant subspace $Y$ of $X$ (together with the action of $G$ on $Y$). $(X,G)$ will always have  minimal nonempty subflows. A point $x\in X$ is \emph{almost periodic} if the closure of its orbit is a minimal subflow.

Let $(X,G)$ and $(Y,G)$ be flows (with the same acting group). A
$G$-homomorphism from $X$ to $Y$ is a continuous map $f: X\longrightarrow Y$  such that $f(gx)=gf(x)$ for all $g\in G$ and $x\in X$.

\begin{Fact}\label{homomorphism of flows}
Let $(X,G)$ and $(Y,G)$ be flows, $f: X\longrightarrow Y$ a $G$-homomorphism. Then $f(x_0)$ is almost periodic whenever $x_0\in X$ is almost periodic.
\end{Fact}

Given a flow $(X,G)$,  its enveloping semigroup $E(X)$  is the closure in the space $X^{X}$ (with the product topology) of the set of maps $\pi_{g}:X\to X$, where $\pi_{g}(x) = gx$, equipped with composition $\circ$ (which is continuous on the left).   So any $e\in E(X)$ is a map from $X$ to $X$.
\begin{Fact}\label{Fact-E(x)}
Let $X$ be a $G$-flow. then
\begin{enumerate}
    \item [(i)] $E(X)$ is a $G$-flow, and $E(E(X))\cong E(X)$;
    \item [(ii)] for any $x\in X$, the closure of its $G$-orbit is exactly $E(X)(x)$. In particular, for any $f\in E(X)$, $E(X)\circ f$ is the closure of $G\cdot f$.
\end{enumerate}
\end{Fact}

By an definable group  $G\subseteq \M^n$, we mean that $G$ is an definable set with an definable map $G\times G\rightarrow G$ as its group operation. For convenience, we assume that $G$ is defined by the formula $G(x)$. We say that $G$ is $A$-definable if $G(x)$ is an $L_A$-formula and the group operation is an $A$-definable map.  For any $M\prec \M$ containing the parameters in $G(x)$, $G(M)=\{g\in M^n| g\in G\}$ is a subgroup of $G$.  It is easy to see that $S_G(M)$ is a $G(M)$-flow with a dense orbit $\{\tp(g/M)|\ g\in G(M)\}$. From now on, we will, through out this paper, assume that every formula $\phi(x)$, with parameters in $\M$, is contained in $G(x)$, namely, the subset $\phi(\M)$ defined by $\phi$ is contained in $G$. Suppose that $\phi(x)$ is an $L_M$-formula and $g\in G(M)$, then the left translate $g\phi(x)$ is defined to be $\phi(g^{-1}x)$. It is easy to check that $(g\phi)(M)=gX$ with $X=\phi(M)$.

\begin{Def}
Let notations be as above.
\begin{itemize}
\item A definable subset  $X\subseteq G$ is \emph{generic} if finitely many left translates of $X$ covers $G$. Namely, there are $g_1,...,g_n\in G$ such that $G=\cup_{i\leq i\leq n} g_iX$.
\item A definable subset  $X\subseteq G$ is \emph{weakly generic} if there is a non-generic definable subset $Y$ such that $X\cup Y$ is generic
\item A definable subset  $X\subseteq G$ is \emph{$f$-generic} if for some/any model $M$ over which $X$ is defined and any $g\in G$, $gX$ does not divide over $M$. Namely, for any $M$-indiscernible sequence $(g_i:i< \omega)$, with $g=g_0$, $\{g_iX: i<\omega\}$ is consistent.
\item A formula $\phi(x)$ is generic if the definable set $\phi(\M)$ is generic. Similarly for weakly generic and $f$-generic formulas.
\item A type $p\in S_G(M)$ is generic if every formula  $\phi(x)\in p$ is generic. Similarly for weakly generic and $f$-generic types.
\item A type $p\in S_G(M)$ is almost periodic if $p$ is a almost periodic point of the $G(M)$-flow $S_G(M)$.
\item A global type $p\in S_G(\M)$ is \emph{strongly $f$-generic} over a small model $M_0$ if every left $G$-translate of $p$ does not fork over $M_0$. A global type $p\in S_G(\M)$ is strongly $f$-generic it it is strongly $f$-generic over some small model.
\end{itemize}
\end{Def}

\begin{Fact}\cite{Newelski}\label{WG sq AG}
Let $AP\sq S_G(M)$ be the space of almost periodic types, and $WG\sq S_G(M)$ the space of weakly generic types. Then $WG=\cl(AP)$.
\end{Fact}

\begin{Fact}\cite{Newelski}\label{Newelski-E(x)-and-closure-of-orbit}
Let $G$ be a group definable in $M$, $N\succ M$ be any $|M|^+$-saturated elementary extension. Let $S_{G,M}(N)$ be the space of all coheirs of types in $S_G(M)$ over $N$. Then enveloping semigroup $E(S_G(M))$ of $S_G(M)$ is isomorphic to $(S_{G,M}(N),*)$ where $*$ is defined as following: for any $p,q\in S_{G,M}(N)$, $p*q=\tp(a\cdot b/N)$ with $a$ realizes $p$ and $b$ realizes $q$, and $\tp(a/N,b)$ is finitely satisfiable in $M$.

\end{Fact}

\subsection{NIP, Definably Amenablity, and Connected component}
Let $G\subseteq \M^n$ be a definable group. Recall that a type-definable over $A$ subgroup $H\subseteq \M^n$ is a type-definable subset of $G$, which is also a subgroup of $G$. We say that $H$ has bounded index if $|G/H|<2^{|T|+|A|}$. For groups definable in $NIP$ structures, the smallest type-definable subgroup $G^{00}$ exist (See \cite{NIP}). Namely, the intersection of all type-definable subgroups of bounded index still has bounded index. We call $G^{00}$ the type-definable connected component of $G$. Another model theoretic invariant is $G^0$, called the definable-connected of $G$,  which is the intersection all definable subgroups of $G$ of finite index. Clearly, $G^{00}\leq G^0$.

Recall also that the  Keisler measure over $M$ on $X$  with $X$ a definable set over $M$, is a finitely additive measure on the Boolean algebra of definable subsets over $M$, subsets of $X$. When we take the monster model, i.e. $M=\M$, we call it a global Keisler measure. A definable group $G$ is said to be definably amenable if it admits a global (left) $G$-invariant probability Keisler measure.
\begin{Fact}\cite{CS}\label{amenable}
Assuming $NIP$. If $G\sq \M^n$ is a definable group. Then the following are equivalent:
\begin{enumerate}
  \item [(i)]  $G$ is definably amenable;
  \item [(ii)]  $G$ admits a global type $p\in S_G(\M)$ with bounded $G$-orbit;
  \item [(iii)]  $G$ admits a strongly $f$-generic type.
\end{enumerate}
\end{Fact}

Moreover,
\begin{Theorem}\cite{CS}\label{f-generic-type-has-bounded-orbit}
For a definably amenable $NIP$ group $G$, we have
\begin{enumerate}
\item [(i)] Weakly generic definable subsets, formulas, and types coincide with $f$-generic definable subsets, formulas, and types, respectively
\item [(ii)] $p\in S_G(\M)$ is $f$-generic if and only if it has bounded $G$-orbit.
\item [(iii)]  $p\in S_G(\M)$ is $f$-generic if and only if it is $G^{00}$-invariant.
\item [(iv)]  A type-definable subgroup $H$ fixing a global $f$-generic type is exactly $G^{00}$
\item [(v)]  $G/G^{00}$ is isomorphic to the Ellis subgroup of $S_G(M)$ for any  $M\prec \M$.
\end{enumerate}
\end{Theorem}

\begin{Rmk}
Note that Theorem \ref{f-generic-type-has-bounded-orbit} (v) gives a positive answer for Ellis group conjecture asked  by Newelski \cite{Newelski} and Pillay \cite{Pillay-TD}.
\end{Rmk}

\begin{Fact}
Suppose that $G\subseteq \M^n$ is a  definably amenable NIP group. Then every $f$-generic type $p\in S_G(N)$ has an $f$-generic global extension.
\end{Fact}
\begin{proof}
It is easy to see that the collection of all weakly generic definable subset of $G$ forms an ideal of the boolean algebra of the collection of all definable subset of $G$. So every weakly generic type $p\in S_G(N)$ has a weakly generic global extension $\bar p\in S_G(\M)$. This completes the proof as weakly generic coincide with $f$-generic.
\end{proof}

\begin{Fact}\label{restriction of ap is ap}\cite{P-Y}
If $p\in S_G(\M)$ (as a $G(\M)$-flow) is almost periodic, then $p|N\in S_G(N)$ (as a $G(N)$-flow)  is almost periodic for any $N\prec \M$.
\end{Fact}

\subsection{Groups definable in $(\Q,+,\times, 0,1)$}
We first give our notations for $p$-adics. By ¡®the $p$-adics¡¯, we mean the field $\Q$. $M$ denotes the structure $(\Q,+,\times, 0,1)$, $\Q^*=\Q\backslash \{0\}$ is the multiplicative group. $\mathbb Z$ is the ordered additive group of integers, the value group
of $\Q$. The group homomorphism $\nu: \Q^*\longrightarrow \mathbb Z$ is the valuation map.  $\M$ denotes a very saturated elementary extension $(\K, +, \times, 0, 1)$ of $M$. $\Gamma_\K$ denotes the value group of $\K$. Similarly, $\K^*=\Q\backslash \{0\}$ is the multiplicative group. We sometimes write $\Q$ for $M$ and $\K$ for $\M$.

For convenience, we use $\Ga$ and $\Gm$ denote the additive group and multiplicative group of field $\M$ respectively.  So $\Ga(M)$ (or $\Ga(\Q)$) and $\Gm(M)$ (or $\Gm(\Q)$) are $(\Q,+)$ and $(\Q^*,\times)$ respectively.

We will be referring a lot to the comprehensive survey \cite{Luc Belair} for the basic model theory of the p-adics. A key
point is Macintyre's theorem \cite{Macintyre} that $\Th(\Q,+,\times, 0,1)$ has quantifier elimination in the language of rings $L_{ring}$ together with new predicates $P_n(x)$ for the $n$-th powers for each $n\in \N^+$. Moreover the valuation is quantifier-free definable in the Macintyre's language $L_{ring}\cap\{P_n|\ n\in \N^+\}$, in particular is definable in the language of rings. (See Section 3.2 of \cite{Luc Belair}.)

The valuation map $\nu$ endows an absolute valuation $|\ |$ on $\Q$: for each $x\in \Q$, $|x|=p^{-\nu(x)}$ if $x\leq 0$ and $|x|=0$ otherwise. The absolute valuation makes $p$-adic field $\Q$  a locally compact topological field, with basis given by the sets $\nu(x-a)\geq n$ for $a\in \Q$ and $n\in \mathbb Z$. For any $X\sq \Q^n$, the ``topological dimension", denoted by $\dim(X)$, is the greatest $k\leq n$ such that the image of $X $ under some projection from $M^n$ to $M^k$ contains an open subset of $\Q^k$. On the other side, as model-theoretic algebraic closure coincides with field-theoretic algebraic closure  (\cite{Hrushovski-Pillay}, Proposition 2.11), we see that for any model of $N$ of $\Th(M)$ the algebraic closure satisfies exchange (so gives a so-called pregeometry on $N$) and there is a finite bound on the sizes of finite sets in uniformly definable families. If $a$ is a finite tuple from $N\models \Th(M)$ and $B$ a subset of $N$ then the algebraic dimension of $a$ over $B$, denoted by $\dim(a/B)$, is the size of a maximal subtuple of a which is algebraically independent over $B$. When  $X\sq \Q^n$  is definable, the algebraic dimension of $X$, denoted by $\alg$-$\dim(X)$, is the maximal $\dim(a/B)$ such that $a\in X$ and $B$ contains the parameters over which  $X$ is defined. It is important to know that when  $X\sq \Q^n$  is definable, then its algebraic-dimension  coincides with its "topological dimension", namely $\dim(X)=\alg$-$\dim(X)$. As a conclusion, for any definable $X\sq \Q^n$, $\dim(X)$ is exactly the algebraic geometric dimension of its Zariski closure.

By a definable manifold $X\sq \Q^n$ over a subset $A\sq \Q$, we mean a  topological space $X$ with a covering by finitely many open subsets $U_1,...,U_m$, and homeomorphisms  of $U_i$ with some definable open $V_i\sq  \Q^n$ for $i=1,...,m$, such that the transition maps are $A$-definable and  continuous. If the transition maps are $C^k$, then we call $X$ a definable $C^k$ manifold over $\Q$ of dimension $n$.   A definable group $G\sq \Q^n$ can be equipped uniquely  with the structure of a definable manifold over K such that the group operation is $C^{\infty}$ (see \cite{Pillay-On fields definable in Qp} and \cite{O-P}). The facts described above work for any $N\models \Th(M)$.

\subsection{$1$-dimensional Groups over $\Q$}
Let $\M=(\K,\times,+,0,1)$ be a saturated extension of $M=(\Q,\times,+,0,1)$, and
\[
(\Gamma_\K,+,<)\succ (\mathbb Z,+,<)
\]
be its valuation group. Recall that $\Gm=(\K^*,\times)$ is the multiplicative group and $\Ga=(\K,+)$ is the additive group. Both $\Gm$ and $\Ga$ are commutative, and hence definably amenable.

\begin{Fact}\label{Gm0 and Ga0}
$\Gm^{00}=\Gm^{0}=\bigcap_{n\in \N^+}P(\K^*)$;  $\Ga^{00}=\Ga^{0}=\Ga$.
\end{Fact}

\begin{Fact}\cite{PPY}\label{complete 1 types over Qp}
The complete $1$-types over $M$ (or $\Q$) are precisely the following:
\begin{enumerate}
  \item [(i)] The realized types $\tp(a/M)$ for each $a\in \Q$.
  \item [(ii)] for each $a\in \Q$ and coset $C$ of $\Gm^0$, the type $p_{a,C}$ saying that $x$ is infinitesimally
close to a (i.e. $\nu(x-a)>n$ for each $n\in \N$), and $(x-a)\in C$.
  \item [(iii)] for each coset $C$ as above the type $p_{\infty, C}$ saying that $x\in C$ and $\nu(x)<n$ for all $n\in \mathbb Z$.
  \item  [(iv)] Any global $1$-type consistent with the partial type $\{\nu(x-a)>\gamma|\ a\in \K,\gamma\in \Gamma_\K\}$  is $\Q$-definable.
  \item  [(v)]  Any global $1$-type consistent with the partial type $\{\nu(x)<\gamma|\  \gamma\in \Gamma_\K\}$  is $\Q$-definable.
\end{enumerate}
\end{Fact}

\begin{Fact}\cite{PPY}\label{complete 1 types over Qp}
Let $p\in S_1(M)$, with the notations as above, then we have
\begin{enumerate}
  \item  [(i)]  $p$ is an $f$-generic/weakly generic type of $S_{\Ga}(M)$ iff $p$ is of form $p_{\infty, C}$.
  \item  [(ii)]  $p$ is an $f$-generic/weakly generic type of $S_{\Gm}(M)$ iff $p$ is of form $p_{\infty, C}$ or $p_{0,C}$.
\end{enumerate}
\end{Fact}
It is well-know that every complete type over $M$ is definable.
\begin{Fact}\cite{PPY}\label{Ga and Gm has dfg}
Let $N\succ M$, $p$  an $f$-generic type/weakly generic type of $S_{\Ga}(N)$ (or $S_{\Gm}(N)$),. Then
\begin{enumerate}
  \item   [(i)]  $p$ is $\emptyset$-definable, which is the unique heir of $p|M$.
  \item   [(ii)]  $p|M$ is an $f$-generic/weakly generic type of $S_{\Ga}(M)$ (or $S_{\Gm}(M)$, respectively).
\end{enumerate}
\end{Fact}

One concludes directly from Fact \ref{Ga and Gm has dfg} that both $\Ga$ and $\Gm$ has \emph{dfg}.

\section{Main results}
\subsection{Closure of the orbit of a global type}
\begin{Lemma}\label{closure-of-orbit}
Let $\mathbb T$ be any fisrt-order theory, and $\M$ a very saturated model of $\mathbb T$. Let $G\subseteq \M^n$ be any definable group and $p\in S_G(\M)$. Then the closure of the $G(\M)$-orbot of $p$ is
\[cl(G(\M)\cdot p)=\{\tp(a\cdot b/\M)|a,b\in G(\bar \M),\ b\models p,\ \text{and}\ \tp(b/\M,a)\ \text{is an heir of}\ \tp(b/\M)\},\]
where $\bar \M$ is some $|\M|^+$-saturated elementary extension of $\M$.
\end{Lemma}
\begin{proof}
Let $S_{G,\M}(\bar \M)$ be the collection of all the coheir extensions of types in $S_G(\M)$. By Fact \ref{Newelski-E(x)-and-closure-of-orbit}, $S_{G,\M}(\bar \M)$ is the envoloping semigroup of $S_G(\M)$,  which is also a $G(\M)$-flow.
Let $\bar p\in S_{G,\M}(\bar \M)$ be any extension of $p$. Then, by Fact \ref{Fact-E(x)}, the closure of $G(\M)$-orbit of $\bar p$ is
\[S_{G,\M}(\bar \M)*\bar p=\{q*\bar p|q\in S_{G,\M}(\bar \M)\},\]
where $q*\bar p=\tp(\bar a\cdot \bar b/\bar \M)$ with $\bar a\models q$, $\bar b\models \bar p$, and $\tp(\bar a/\bar \M,\bar b)$ is finitely satisfiable in $\M$.

It is easy to see that $\pi: S_{G,\M}(\bar \M)\lra S_G(\bar \M)$ defined by $p\mapsto p|\M$ is a $G(\M)$-homomorphism.
We claim that
\begin{Claim}
$\cl(G(\M).p)= \pi(\cl(G(\M).\bar p))$
\end{Claim}
\begin{proof}
Since both $S_{G,\M}(\bar \M)$  and $S_G(\M)$ are compact and Hausdoff, we see that $\pi(\cl(G(\M).\bar p))$ is closed. Since $G(\M).p\sq \pi(\cl(G(\M).\bar p))$, we have
\[\cl(G(\M).p)\sq\pi(\cl(G(\M).\bar p)).\]
Conversely, if $\bar q\in \cl(G(\M).\bar p)$, then for any $L_{\bar \M}$-formula $\phi\in \bar q$ there is $g\in G(\M)$ such that $\phi\in g\bar p$. Take any $L_\M$-formula $\psi\in \pi(\bar q)=\bar q|\M$, there is $g\in G(\M)$ such that $\psi\in g\bar p$, thus $\psi\in g\bar p|\M= gp$. So $\pi(\bar q)\in \cl(G(\M)\cdot p)$ as required.
\end{proof}

By the above Claim, we conclude that
\[cl(G(\M)\cdot p)\subseteq \{\tp(a\cdot b/\M)|a,b\in G(\bar \M),\ b\models p,\ \text{and}\ \tp(b/\M,a)\ \text{is an heir of}\ \tp(b/\M)\}.\]

On the other side, suppose that $a,b\in G(\bar \M)$ and $b\models p$ such that $\tp(a/\M,b)$ is finitely satisfiable in $\M$. We now show that $\tp(a\cdot b/\M)\in \cl(G(\M)\cdot p)$.

Let $\bar p=\tp(\bar b/\bar \M)\in S_{G,\M}(\bar \M)$ be an extension of $p$ and $\tp(\bar a/\bar \M)\in S_{G,\M}(\bar \M)$ an extension of $\tp(a/\M,b)$ such that $\tp(\bar a/\bar \M, \bar b)$ is finitely satisfiable in $\M$. Then
\[\tp(\bar a\cdot \bar b/\bar \M)=\tp(\bar a/\bar \M)*\tp(\bar b/\bar \M)\in \cl(G(\M)\cdot \bar p),\]
and
\[\pi(\tp(\bar a\cdot \bar b/\bar \M))=\tp(\bar a\cdot \bar b/\M)\in cl(G(\M)\cdot p).\]
Let $\phi(x,y)\in {\cal L}_{\M}$, and ${\cal N}\succ \bar \M$ be any $|{\bar \M}|^+$-saturated model, we see that
\[{\cal N}\models \phi(a,b)\Longrightarrow \phi(x,b)\in \tp(a/\M,b) \Longrightarrow \phi(x,b)\in \tp(\bar a/\bar \M,\bar b).\]
Since $\tp(\bar a/\bar \M,\bar b)$ is finitely satisfiable in $\M$, $\tp(\bar a/\bar \M,\bar b)$ is $\M$-invariant. As $\tp(b/\M)=\tp(\bar b/\M)$, we have $\phi(x,\bar b)\in \tp(\bar a/\bar \M,\bar b)$, and hence ${\cal N}\models \phi(\bar a,\bar b)$. This implies that $\tp(a,b/\M)=\tp(\bar a,\bar b/\M)$. So \[
tp(a\cdot b/\M)=\tp(\bar a\cdot \bar b/\M)\in cl(G(\M)\cdot p)\]
as required.
\end{proof}

\begin{Lemma}\label{every heir f-generic imp almost periodic}
Let $\M$ be a saturated model with $NIP$, and $G\subseteq \M^n$ is a definably amenable group. If every heir of $p\in S_G(\M)$ is $f$-generic. Then $G(\M)\cdot p$ is closed. In particular, $p$ is almost periodic.
\end{Lemma}
\begin{proof}
Let $\bar\M\succ \M$ be the any $|\M|^+$-saturated model. By Lemma \ref{closure-of-orbit}, $cl(G(\M)\cdot p)$ is the collection of types of form $\tp(a\cdot b/\M)$ with $a,b\in G(\bar\M)$,  $b\models p$, and $\tp(b/\M,a)$ is a heir of $p$. Let $\tp(\bar b/\bar \M)$ be an heir extension of $p$ such that
\[
\tp(b/\M,a)\subseteq \tp(\bar b/\bar \M).
\]
Since $\tp(\bar b/\bar \M)$ is $f$-generic, it is $G^{00}(\bar \M)$-invariant. Let $a'\in G(\M)$ such that
\[a'G^{00}(\bar \M)=aG^{00}(\bar \M),\]
we have $\tp(a\cdot \bar b/\bar \M)=\tp(a'\cdot \bar b/\bar\M)$. This implies that
\[\tp(a\cdot b/\M)=\tp(a\cdot \bar b/\M)=\tp(a'\cdot \bar b/\M)=a'\cdot \tp(\bar b/\M)=a'\cdot p\in G(\M)\cdot p.\]
\end{proof}

\subsection{$G^{00}$ and global definable $f$-generic types}
We now assume that $M=(\Q,+,\times,0,1)$, $\M=(\K,+,\times,0,1)$ is a very saturated elementary extension of $M$, and $\Gamma$ is the valuation group of $\K$. Every definable group $G$ is defined in the saturated model $\M$, and defined by the formula $G(x)$ with parameters from $\Q$. For any $N\succ M$, $G(N)$ is the realizations of $G(x)$ in $N$, which is a definable group in $N$. We say that a type $p\in S_G(M)$ is  $G(M)$-invariant type if $gp=p$ for all $g\in H(M)$. It is easy to see that every heir of $p$ over any model $N\succ M$ is $G(N)$-invariant. We say that $G$ has a $G$-invariant type if there is $p\in S_G(M)$ which is $G(M)$-invariant. By Theorem \ref{f-generic-type-has-bounded-orbit},  $G$ has a  $G$-invariant type if and only if $G$ is definably amenable and $G^{00}=G$.

\begin{Lemma}\label{intergal-f-generic-type}
Let $G\subseteq \M^n$, $H\subseteq \M^k$, and $T\subseteq \M^l$ be definable groups such that $H\unlhd G$ and
\[
1\longrightarrow H  \xlongrightarrow {i}  G  \xlongrightarrow {\pi}  T \longrightarrow 1
\]
is a short exact sequence.  Suppose that both $H$ and $T$ have \emph{dfg}, and $H$ has an $H$-invariant type. Then $G$ has \emph{dfg}. Moreover $H\leq G^{00}$ and $\pi(G^{00})={T}^{00}$.
\end{Lemma}
\begin{proof}
Let $p\in S_H(\M)$ be an f-generic type of $H$,   and $q\in S_{T}(\M)$ a f-generic type of $T$, both of them definable over a small submodel $N$. We now define a type $r\in S_G(\M)$ as follows: for any $ L_\M$-formula $\phi(x)$,
\[
\phi(x)\in r\Longleftrightarrow \{\pi(g)|g\in G,\  g^{-1}\phi(x)\in p\}\in q
\]
Since $p$ is definable,
\[
D_\phi=\{g|g\in G,\  g^{-1}\phi(x,b)\in p\}
 \]
is a definable subset of $G$, hence $\pi(D_\phi)$ is a definable subset of $T$, and we identify it with a formula defining it.
Since $p$ is $H$-invariant, we see that $D_\phi=HD_\phi$ and thus $T\backslash \pi(D_\phi)=\pi(D_{\neg\phi})$.
 So $r$ is well-defined and  complete. Since $q$ is also definable over $N$, $r$ is definable over $N$.

 Let $\bar \M$ be an $|\M|^+$-saturated elementary extension of $\M$, and $\tilde{\M}$ an $|\bar \M|^+$-saturated elementary extension of $\bar \M$ . For any $\bar g\in G(\bar \M)$ such that $\pi(\bar g)$ realizes $q$ and any $\tilde h\in H(\tilde\M)$ such that $\tilde h$ realizes $ p|\bar \M$ (the unique heir of $p$ over $\bar \M$), we claim that:
 \begin{Claim}
 $r=\tp(\bar g\tilde h/\M)$.
 \end{Claim}
 \begin{proof}
 Let $\phi(x)\in r$, and $\psi(x,y)$ be the formula $y^{-1}\phi(x)$. As $p$ is definable, there is a formula $\theta(y)$ such that $\psi(x,y)\in p$ iff $\theta(y)$ holds. Since $\pi(\theta(y))\in q$, we see that $\theta(\bar g)$ holds in $\bar \M$ ans thus $\psi(\bar g,x)\in p|\bar \M$. So $\psi(\bar g,\tilde h)$ holds in $\tilde \M$, which means that $\phi(\bar g\tilde h)$ holds as required.
 \end{proof}

 \begin{Claim}
 $r$ is invariant under $\pi^{-1}({T}^{00})$.
 \end{Claim}
  \begin{proof}
Let $r=\tp(\bar g\tilde h/\M)$, with $\bar g$ and $\tilde h$ as above Claim. By the above Claim, we only need to check that for any $g\in \pi^{-1}({T}^{00})$, $\pi(g\bar g)$ realizes $q$. Since
$\pi(g\bar g)=\pi(g)\pi(\bar g)$ and $q$ is ${T}^{00}$, $q=\tp(\pi(g)\pi(\bar g)/\M)=\tp(\pi(g \bar g)/\M)$ as required.
\end{proof}
Since ${T}^{00}$ has bounded index in $T$, we see that $\pi^{-1}({T}^{00})$ has bounded index in $G$. So $r$ has a bounded orbit and hence is an $f$-generic type of $G$ by Theorem \ref{f-generic-type-has-bounded-orbit} (ii). This  also implies that $G^{00}=\pi^{-1}({T}^{00})$ by Theorem \ref{f-generic-type-has-bounded-orbit} (iv). So we conclude that $H\leq G^{00}$ and $\pi(G^{00})={T}^{00}$.

\end{proof}

Recall that a connected algebraic group $G$ is trigonalizable over $\Q$ if it  admits a short exact sequence
\[
1\longrightarrow U  \xlongrightarrow {i}  G  \xlongrightarrow {\pi}  T \longrightarrow 1,
\]
where $U\subseteq \M^{m}$ is the maximal unipotent subgroup of $G$ and $T\subseteq \M^{n}$ is a split torus, namely, isomorphic to  $ \Gm^{n}$. $U$ is connected as char $\Q=0$, and a connected  unipotent group admits a normal sequence
\[
1=U_0\leq...\leq U_i\leq U_{i+1}\leq U_k=U
\]
such that each $U_{i+1}/U_i$ has dimension one for each $i\leq k$. Since any connected one-dimensional unipotent algebraic group is isomorphic to $\Ga$, we see that $U_{i+1}/U_i\cong \Ga=(\K,+)$ for each $i\leq k$, and thus $U$ splits over $\Q$. By \cite{iAG200}, Theorem 17.26, the short exact sequence
\[
1\longrightarrow U  \xlongrightarrow {i}  G  \xlongrightarrow {\pi}  T \longrightarrow 1,
\]
splits. So $G$ is definably isomorphic to the semiproduct $T\ltimes U$

Induction on $\dim(U)$ we could easily conclude that

\begin{Cor}\label{unipotent algebraic group is definably-connected}
If $U$ is an unipotent algebraic group, then $U^{00}=U$.
\end{Cor}

\begin{Lemma}\label{intergal-f-generic-type-of-direct-product}
Let $G$ and $H$ be definable groups such that both $G$ and $H$ have \emph{dfg}. Then the direct product $G\times H$ has \emph{dfg}. Moreover $(G\times H)^{00}=G^{00}\times H^{00}$.
\end{Lemma}
\begin{proof}
Let $p=\tp(g/\M)\in S_G(\M)$ be an f-generic type of $H$ definable over $M$ and $q=\tp(h/\bar \M)\in S_{H}(\bar \M)$ a f-generic type of $H$ definable over $M$, where $\bar \M$ is an $|\M|^+$-saturated elementary extension of $\M$. Let $r=\tp((g,h)/\M)\in S_{G\times H}(\M)$. An analogues argument as Lemma \ref{intergal-f-generic-type} showing that $r$ is definable, and  for any $g'$ realizes $p$ and any $h'$ realizes $q|\bar \M$, we have $r=\tp((g',h')/\M)$. Since $p$ is $G^{00}$-invariant and $q$ is $H^{00}$-invariant, we see that $r$ is $G^{00}\times H^{00}$-invariant. Since $G^{00}\times H^{00}$ has bounded index in $G\times H$, we see that $r$ has bounded index and hence an $f$-generic type by Theorem \ref{f-generic-type-has-bounded-orbit} (ii), and $G^{00}\times H^{00}=(G\times H)^{00}$ by Theorem \ref{f-generic-type-has-bounded-orbit} (iv).
\end{proof}

By Fact \ref{Gm0 and Ga0} and  Lemma \ref{intergal-f-generic-type} we conclude directly that

\begin{Cor}\label{Gm-n-is-dfg-Gmn00}
$\Gm^n$ has \emph{dfg}, and ${\Gm^n}^{00}={\Gm^{00}}^n={\Gm^{0}}^n$.
\end{Cor}
Moreover, we have
\begin{Fact}\label{f-g is defble}\cite{Yao-Presburger}
Let $N\succ M$. If $p\in S_{\Gm^n}(N)$ is $f$-generic, then $p$ is definable over $\emptyset$.
\end{Fact}

We conclude directly from Lemma \ref{intergal-f-generic-type}, Corollary \ref{unipotent algebraic group is definably-connected}, and Corollary \ref{Gm-n-is-dfg-Gmn00}, we conclude directly that

\begin{Cor}\label{G-is-dfg-G00}
If $G$ is trigonalizable over $\Q$, then $G$ has \emph{dfg}. Moreover, if $G\cong T\ltimes U$ with $U$ the maximal unipotent subgroup of $G$ and  $T\cong \Gm^n$  the tours, we have
\[G^{00}\cong T^{00}\ltimes U=T^{0}\ltimes U.\]
\end{Cor}

\subsection{f-generic types and almost periodic}

We now suppose that  $G$ is trigonalizable algebraic group  over $\Q$, which is isomorphic to the semiproduct $\Gm^{n}\ltimes U$, with $U\subseteq \M^{m}$ an unipotent algenraic group over $\Q$. We now denote $\Gm^n$ by $T$ and  identify $G$ with $T\ltimes U $, and every element $g\in G$ with the unique pair $(t,u)$, where $t\in T$, $u\in U$, and $g=tu$. It is easy to see that the group operation is given by
\[
(t_1,u_1)(t_2,u_2)=t_1t_2 t_2^{-1} u_1 t_2u_2=(t_1t_2,t_2^{-1} u_1 t_2u_2)=(t_1t_2, u_1^{t_2}u_2).
\]
The groups $G$, $T$, and $U$ are defined by the formulas $G(x,y)$, $T(x)$, and $U(y)$ respectively.




For any $t=(t_1,...,t_n)\in T$, and $m=(m_1,...,m_n)\in {\mathbb Z}^n$, by $t^m$  we mean $\Pi_{i=1}^{n}t_i^{m_i}$. For any $m=(m_1,...,m_n)$ and $m'=(m'_1,...,m'_n)$ in ${\mathbb Z}^n$, by $m-m'$ we mean $(m_1-m'_1,...,m_n-m'_n)$. Let $\Gamma_1\succ\Gamma_0\succ(\mathbb Z,+,<)$, we say that $\gamma\in \Gamma_1$ is bounded over $\Gamma_0$ if there are $c,d\in \Gamma_0$ such that $c\leq \gamma\leq d$ and unbounded if otherwise.

\begin{Lemma}\label{tp(u,t) f-generic-II}
Let $N=(K,+,\times,0,1)$ be any elementary extension of $M$, and $\Gamma_K$ is the valuation group of $K$. If $t\in T$ and $u\in U$ such that $\tp(t,u/N)$ is an $f$-generic type of $G$ over $N$, then
\begin{enumerate}
    \item [(i)] $\tp(t,u/N)=\tp(t,b^tu/N)$ for For any $b\in U(N)$;
    \item [(ii)] For any $m\in {\mathbb Z}^n$, $b\in U(N)$,  and definable function $f$ definable over $N$, $\nu(t^m)+\nu(f(b^tu)) $ is unbounded over $\Gamma_K$ whenever $m\neq (0,...,0)$.
\end{enumerate}
\end{Lemma}
\begin{proof}
\begin{enumerate}
    \item [(i)] Since any $f$-generic extension $\bar p\in S_G(\M)$ of $\tp(t,u/N)$ is $U$-invariant, we see that $\tp(t,u/N)$ is $U(N)$-invariant, and hence
    \[\tp(t,b^tu/N)=(1,b)\cdot\tp(t,u/N)=\tp(t,u/N);
    \]
    \item [(ii)] Suppose for a contradiction that there are $c,d\in \Gamma_K$, $m\leq 0\in {\mathbb Z}^n$, and function $f$ definable over $N$ such that
    \[c\leq \nu(t^m)+\nu(f(b^tu))\leq d.\]
     Then for any $a\in T^{00}$, we see that
     \[\tp(t,b^tu/N)=\tp(t,u/N)=\tp(at,b^tu/N).\]
     So we have
     \[
     c\leq \nu({t}^m)+\nu(f(b^{t}u))\leq d,\ \text{and}\ \ c\leq \nu({at}^m)+\nu(f(b^{t}u))\leq d.
     \]
    If $m\neq 0$, then we can take  some $a\in T^{00}$ such that $\nu(a^m)=\sum_{i=1}^{n}m_i\nu(a_i)>d-c$. But
    \[
    c\leq \nu({at}^m)+\nu(f(b^{t}u))=\nu(a^m)+\nu({t}^m)+\nu(f(b^{t}u))\leq d
    \]
     implies that $\nu(a^m)\leq d-c$. A contradiction.
\end{enumerate}
\end{proof}

\begin{Lemma}\label{tp(u,t) f-generic-III}
Let $N=(K,+,\times,0,1)$ be any elementary extension of $M$, and $\Gamma_K$ is the valuation group of $K$. If $t\in T$ and $u\in U$ such that
\begin{enumerate}
    \item [(i)] $p(x,y)=\tp(t,u/N)=\tp(t,b^tu/N)$ for For any $b\in U(N)$
    \item [(ii)] For any $m\in {\mathbb Z}^n$, $b\in U(N)$,  and definable function $f$ definable over $N$, $\nu(t^m)+\nu(f(b^tu)) $ is unbounded over $\Gamma_K$ whenever $m\neq (0,...,0)$.
\end{enumerate}
Then $\tp(t,u/N)$ is an $f$-generic type of $G$ over $N$.
\end{Lemma}
\begin{proof}
It is easy to see that condition (i) holds iff for any $\phi(x,y)$ and any $b\in U(N)$
\begin{equation}\label{condition 1}
(\phi(x,y)\leftrightarrow \phi(x,b^{x}y))\in p(x,y),
\end{equation}
and  condition (ii) holds iff for any $N$-definable function $f$, any $b\in U(N)$,  any $m\neq (0,...,0)$, and any $L_N$-formual $\phi(x,y)\in p$
\begin{equation}\label{condition 2}
\begin{split}
     & \text{either}\ \ \  N\models \forall z\bigg((z\neq 0)\rightarrow \exists x, y\big(\phi(x,y)\wedge(\nu(x^m)+\nu(f(x^{-1}bxy))>\nu(z))\big)\bigg)\ \  \\
     &\text{or}\ \ \ \ \ \ \ \  N\models \forall z\bigg((z\neq 0)\rightarrow \exists x, y\big(\phi(x,y)\wedge(\nu(x^m)+\nu(f(x^{-1}bxy))<\nu(z))\big)\bigg).
\end{split}
\end{equation}

Let $\tp(\bar t,\bar u/\M)$ be an heir of $\tp(t,u/N)$ over $\M$, where $\bar t\in T(\bar \M)$ and $\bar u\in U(\bar \M)$ with $\bar \M=(\bar \K,+,\times,0,1)$ an $|\M|^+$-saturated extension of $\M$. Then, by the definition of heir,  it is easy to see that (\ref{condition 1}) and (\ref{condition 2}) hold when we replace $t,u,N$ by $\bar t,\bar u,\M$ respectively. So condition (i) and (ii) hold if we replace $t,u,N$ by $\bar t,\bar u,\M$ respectively.

Now it suffices to show that $\tp(\bar t,\bar u/\M)$ is an $f$-generic type of $G$ over $\M$. Let $(a,b)\in G^{00}$, we now show  that
\[(a,b)\tp(\bar t,\bar u/\M)=\tp(a\bar t,b^{\bar t}\bar u/\M)=\tp(\bar t,\bar u/\M).\]
By quantifier elimination, we only need to check that for every polynomial $g(x,y)\in \K[x,y]$, and every $n$-th power $P_n$, we have
\begin{equation}\label{equ-iff}
\bar \M\models  P_n(g(\bar t,\bar u))\iff \bar \M\models P_n(g(a\bar t,b^{\bar t}\bar u)).
\end{equation}

Suppose that $g(x,y)=\sum_{i=1}^{k}g_i(y)x^{m_i}$, where $g_i\in \K[y]$, and $m_i=(m_{i1},...,m_{in})\in {\mathbb N}^n$.    Now $m_{i}\neq m_{j}$ whenever $i\neq j\leq k$. By condition (ii), we have
\[
\nu(\frac{g_j({b}^{\bar t}\bar u) {\bar t}^{m_j}}{g_i({b}^{\bar t}\bar u) {\bar t}^{m_i}})=\nu(\frac{g_j({b}^{\bar t}\bar u)}{g_i({b}^{\bar t}\bar u)}{\bar t}^{(m_j-m_i)})=\nu(\frac{g_j({b}^{\bar t}\bar u)}{g_i({b}^{\bar t}\bar u)})+\nu({\bar t}^{(m_j-m_i)})
\]
is unbounded over $\Gamma_\K$ for $i\neq j$. So there is unique ${j^*}\leq k$ such that
\begin{equation}\label{minimal-term-I'}
\nu(g_{j^*}({b}^{\bar t}\bar u) {\bar t}^{m_{j^*}})< \nu(g_i({b}^{\bar t}\bar u) {\bar t}^{m_i})+\Gamma_\K
\end{equation}
for all $i\neq {j^*}$. Since $\nu(a^m)\in \Gamma_\K$ is bounded for all $m\in {\mathbb Z}^{n}$, and
\[
\nu(g_{i}({b}^{\bar t}\bar u) {(a\bar t)}^{m_{i}})=\nu(g_{i}({b}^{\bar t}\bar u) {a}^{m_{i}}{\bar t}^{m_{i}})=\nu(g_{i}({b}^{\bar t}\bar u) {\bar t}^{m_{i}})+\nu({a}^{m_{i}})
\]
for all $i\leq k$, we see that
\begin{equation}\label{minimal-term-II}
\nu(g_{j^*}({b}^{\bar t}\bar u) {(a\bar t)}^{m_{j^*}})< \nu(g_i({b}^{\bar t}\bar u) {(a\bar t}^{m_i}))+\Gamma_\K
\ \ \    \text{for all}\  i\neq {j^*}.
\end{equation}
By condition (i) and (\ref{minimal-term-I'}), we have
\begin{equation}\label{minimal-term-I}
\nu(g_{j^*}(\bar u){\bar t}^{m_{j^*}})< \nu(g_i(\bar u){\bar t}^{m_i})+\Gamma_\K\ \ \    \text{for all}\  i\neq {j^*}.
\end{equation}
Moreover, $g_{j^*}(\bar u){\bar t}^{m_{j^*}}$ and $g_{j^*}(b^{\bar t}\bar u){\bar t}^{m_{j^*}}$ are in the same coset of $P_n({\bar \K}^*)$ as $\tp(g_{j^*}(\bar u){\bar t}^{m_{j^*}}/\M)=\tp(g_{j^*}(b^{\bar t}\bar u){\bar t}^{m_{j^*}}/\M)$. Let $\lambda\in \Q^*$ such that
\begin{equation}\label{same-coset}
\bar \M\models \lambda P_n(g_{j^*}(\bar u){\bar t}^{m_{j^*}}) \ \text{and}\  \ \bar \M\models  \lambda P_n(g_{j^*}(b^{\bar t}\bar u){\bar t}^{m_{j^*}}).
\end{equation}
We see that
\[
\bar \M\models  P_n(g(\bar t,\bar u))\iff \bar \M\models \lambda^{-1} P_n(\frac{g(\bar t,\bar u)}{g_{j^*}(\bar u){\bar t}^{m_{j^*}}})
\]
By (\ref{minimal-term-I}),
\[
\frac{g(\bar t,\bar u)}{g_{j^*}(\bar u){\bar t}^{m_{j^*}}}=1+\mu,
\]
with $\nu(\mu)>\Gamma_\K$. Namely $\frac{g(\bar t,\bar u)}{g_{j^*}(\bar u){\bar t}^{m_{j^*}}}$ is infinitesimal close to $1$ over $\K$. As $P_n(\Q^*)$ is open and $1\in P_n(\Q^*)$, there is a $\Q$-definable open neighborhood $\psi(\Q)$ of $1$ which is contained in $P_n(\Q^*)$. Clearly, $1+\mu\in \psi({\bar \K})\subseteq P_n({\bar \K}^*)$,  So we conclude that
\begin{equation}\label{Pn(lambda)-I}
\bar \M\models  P_n(g(\bar t,\bar u))\iff \bar \M\models \lambda^{-1} P_n(1+\mu)\iff \bar \M\models  P_n(\lambda)
\end{equation}
By (\ref{minimal-term-II}) and (\ref{same-coset}), a similarly argument showing that
\begin{equation}\label{Pn(lambda)-II}
\M\models P_n(g(a\bar t,b^{\bar t}\bar u))\iff \bar \M\models  P_n(\lambda).
\end{equation}
Now (\ref{Pn(lambda)-I}) and (\ref{Pn(lambda)-II}) implies (\ref{equ-iff}) as required.
\end{proof}

By Lemma \ref{tp(u,t) f-generic-II} and Lemma \ref{tp(u,t) f-generic-III}, we conclude directly that
\begin{Prop}\label{f-generics iff}
Let $N=(K,+,\times,0,1)$ be any elementary extension of $M$, and $\Gamma_K$ is the valuation group of $K$. If $t\in \Gm^{n}(\M)$ and $u\in U(\M)$. Then $\tp(t,u/N)$ is an $f$-generic type of $G$ if and only if the followings hold.
\begin{enumerate}
    \item [(i)] $\tp(t,u/N)=\tp(t,b^tu/N)$ for For any $b\in U(N)$
    \item [(ii)] For any $m\in {\mathbb Z}^n$, $b\in U(N)$,  and definable function $f$ definable over $N$, $\nu(t^m)+\nu(f(b^tu))$ is unbounded over $\Gamma_K$ whenever $m\neq (0,...,0)$.
\end{enumerate}
\end{Prop}

\begin{Cor}\label{every heir is f-generic}
Let $N=(K,+,\times,0,1)$ be any elementary extension of $M$. If $p\in S_G(N)$ is $f$-generic, then every heir of $p$ over any $N'\succ N$ is $f$-generic.
\end{Cor}
\begin{proof}
Let $p(x,y)=\tp(t,u/N)$ with $t\in T$ and $u\in U$.  By  Proposition \ref{f-generics iff}, we have
\begin{enumerate}
    \item [(i)] $\tp(t,u/N)=\tp(t,b^tu/N)$ for For any $b\in U(N)$
    \item [(ii)] For any $m\in {\mathbb Z}^n$, $b\in U(N)$,  and definable function $f$ definable over $N$, $\nu(t^m)+\nu(f(b^tu))$ is unbounded over $\Gamma_K$ whenever $m\neq (0,...,0)$.
\end{enumerate}

Let $N'\succ N$, $\bar t\in T$ and $\bar u\in U$ such that $\tp(\bar t,\bar u/N')$ is an heir of $p$ over $N'$. Then, as we proved in Lemma \ref{tp(u,t) f-generic-III}, conditions (i) and (ii) hold when we replace $t,u,N$ by $\bar t,\bar u,N'$ respectively. So $\tp(\bar t,\bar u/N')$ is $f$-generic over $N'$ by Proposition \ref{f-generics iff}.
\end{proof}

\begin{Theorem}\label{Main Thm}
Let $N=(K,+,\times,0,1)$ be any elementary extension of $M$. Then $p\in S_G(N)$ is $f$-generic if and only if $p$ is almost periodic.
\end{Theorem}
\begin{proof}
Clearly, every almost periodic type is $f$-generic by Fact \ref{WG sq AG}.

Let $p=\tp(t,u/N)$ be an $f$-generic type of $G$ over $N$. Let  $p'\in S_G(\M)$ be any heir of $p$. Then, By Corollary \ref{every heir is f-generic}, $p'$ is $f$-generic and every heir of $p'$ is $f$-generic. By Lemma \ref{every heir f-generic imp almost periodic}, $p'$ is almost periodic. By Fact \ref{restriction of ap is ap}, $p=p'|N$ is almost periodic.
\end{proof}

\begin{Question}
Let $U$ and $T$ are as above. By Fact \ref{f-g is defble}, every $f$-generic type of $T$ over any model $N$ is definable over $\emptyset$. Is it ture that every $f$-generic type of $U$ is definable over $\emptyset$?
\end{Question}


\newpage









\begin{thebibliography}{99}
\bibitem{Auslander} J.Auslander, {\em Minimal flows and their extensions}, North Holland, Amsterdam, 1988.



\bibitem{Luc Belair} Luc Belair, Panorama of p-adic model theory, Ann. Sci. Math. Quebec, 36(1), 2012


\bibitem{CS} A. Chernikov and P. Simon, Model theoretic tame dynamics, preprint.



\bibitem{NIP} E. Hrushovski, Y. Peterzil, and A. Pillay, Groups, measures, and the NIP, Journal AMS 21 (2008), 563-596.

\bibitem{Hrushovski-Pillay} E. Hrushovski and A. Pillay, Groups definable in local fields and pseudo-finite fields, 85(1994), 203-262.





\bibitem{G. Jagiella} G. Jagiella, Definable topological dynamics and real lie groups, Math. Logic Quarterly, to appear.

\bibitem{Macintyre} A. Macintyre, On definable subsets of p-adic fields, J. Symbolic Logic, 41(1976), 605�C610.

\bibitem{M-book} D. Marker, {\em Model Theory: An Introduction}, Spinger-Verlag, NY, 2002.

\bibitem {iAG200} J.L. Milne, {\em Algebraic groups-The theory of group schemes of finite type over a field}, Cambridge University Press, 2017.

\bibitem{Newelski} L. Newelski, Topological dynamics of definable group actions, J. Symbolic Logic,  74(2009), 50-72.


\bibitem{N-P} L. Newelski and M. Petrykowski, Weak generic types and coverings of groups I, Fundamenta Mathematicae 191(2006), 201-225.


\bibitem{O-P} A. Onshuus and A. Pillay,  Definable Groups And Compact P-Adic Lie Groups, Journal of the London Mathematical Society, 78(1), (2008) , 233-247.

\bibitem{PPY} D. Penazzi, A. Pillay
, and N. Yao, Some model theory and topological dynamics of p-adic algebraic groups, Arxiv.

\bibitem{Pillay-On fields definable in Qp} A. Pillay,  On fields definable in $\Q$, Archive Math. Logic 29 (1989) 1-7.

\bibitem{Pillay-TD} A. Pillay,  Topological dynamics and definable groups, J. Symbolic Logic, 78(2013), 657-666




















\bibitem {P-Y2} A. Pillay and N. Yao, On groups over $\Q$ with definable $f$-generic types, preprint.

\bibitem{P-Y} A. Pillay and N. Yao, On minimal flows, definably amenable groups, and o-minimality, Adv. in Mathematics, 290(2016), 483-502.












%
\bibitem{Pzt-book} B. Poizat, {\em A course in model theory}, Spinger-Verlag, NY, 2000.
















\bibitem{Yao-Presburger} N. Yao, On f-generic types in Presburger Arithmetic, Studies in Logic, submitted.

\bibitem{Y-L} N. Yao and D. Long, Topological dynamics for groups definable in real closed field, Annals of Pure and Applied Logic, 166(2015), 261-273.
\end{thebibliography}
\end{document}